\UseRawInputEncoding
\documentclass[letterpaper,10pt,reqno]{amsart}
\usepackage{indentfirst} % Indent first paragraph after section header
\usepackage{amssymb}
\usepackage{mathtools} % improves amsmath
\usepackage{mathabx} % makes many symbols (as $\leq$) more beautiful
\usepackage{amsthm}
\usepackage{thmtools}
\usepackage{enumitem} % special itemize with custom tags and labels that work
\usepackage[colorlinks=true]{hyperref} % Links in the pdf. Backref option: each bibliographical entry denotes where it was cited
\usepackage[usenames,dvipsnames]{xcolor} % With XCOLOR, printed quality is good, (but with COLOR it's bad)
\usepackage{ifthen}
\usepackage{tikz}
\usetikzlibrary{decorations.pathreplacing}
\usepackage{tikz-cd}
\usepackage{verbatim}

\usepackage{pgfplots}
\pgfplotsset{compat=newest}

\usetikzlibrary{calc, positioning}
\usepackage{pgfplots}
\pgfplotsset{width=10cm,compat=1.9}
\usepgfplotslibrary{external}
  
\makeatletter

\def\MRbibitem{\@ifnextchar[\my@lbibitem\my@bibitem}

\def\mybiblabel#1#2{\@biblabel{{\hyperref{http://www.ams.org/mathscinet-getitem?mr=#1}{}{}{#2}}}}

\def\myhyperanchor#1{\Hy@raisedlink{\hyper@anchorstart{cite.#1}\hyper@anchorend}}

\def\my@lbibitem[#1]#2#3#4\par{%
  \item[\mybiblabel{#2}{#1}\myhyperanchor{#3}\hfill]#4%
  \@ifundefined{ifbackrefparscan}{}{\BR@backref{#3}}%
  \if@filesw{\let\protect\noexpand\immediate% write to aux-file
    \write\@auxout{\string\bibcite{#3}{#1}}}\fi\ignorespaces%
}

\def\my@bibitem#1#2#3\par{%
  \refstepcounter\@listctr% standard tex item counter for the generic item number
  \item[\mybiblabel{#1}{\the\value\@listctr}\myhyperanchor{#2}\hfill]#3%
  \@ifundefined{ifbackrefparscan}{}{\BR@backref{#2}}%
  \if@filesw\immediate\write\@auxout% write to aux-file
    {\string\bibcite{#2}{\the\value\@listctr}}\fi\ignorespaces%
}

\makeatother
\DeclareFontFamily{U} {MnSymbolA}{}
\DeclareFontShape{U}{MnSymbolA}{m}{n}{
   <-6> MnSymbolA5
   <6-7> MnSymbolA6
   <7-8> MnSymbolA7
   <8-9> MnSymbolA8
   <9-10> MnSymbolA9
   <10-12> MnSymbolA10
   <12-> MnSymbolA12}{}
\DeclareFontShape{U}{MnSymbolA}{b}{n}{
   <-6> MnSymbolA-Bold5
   <6-7> MnSymbolA-Bold6
   <7-8> MnSymbolA-Bold7
   <8-9> MnSymbolA-Bold8
   <9-10> MnSymbolA-Bold9
   <10-12> MnSymbolA-Bold10
   <12-> MnSymbolA-Bold12}{}
\DeclareSymbolFont{MnSyA} {U} {MnSymbolA}{m}{n}
 \DeclareFontFamily{U} {MnSymbolC}{}
\DeclareFontShape{U}{MnSymbolC}{m}{n}{
  <-6> MnSymbolC5
  <6-7> MnSymbolC6
  <7-8> MnSymbolC7
  <8-9> MnSymbolC8
  <9-10> MnSymbolC9
  <10-12> MnSymbolC10
  <12-> MnSymbolC12}{}
\DeclareFontShape{U}{MnSymbolC}{b}{n}{
  <-6> MnSymbolC-Bold5
  <6-7> MnSymbolC-Bold6
  <7-8> MnSymbolC-Bold7
  <8-9> MnSymbolC-Bold8
  <9-10> MnSymbolC-Bold9
  <10-12> MnSymbolC-Bold10
  <12-> MnSymbolC-Bold12}{}
\DeclareSymbolFont{MnSyC} {U} {MnSymbolC}{m}{n}

\DeclareMathSymbol{\top}{\mathord}{MnSyA}{219} % smaller symbol for transpose
\DeclareMathSymbol{\plus}{\mathord}{MnSyC}{20} % a smaller plus sign

\declaretheorem[numberwithin=section]{theorem}
\declaretheorem[sibling=theorem]{lemma}
\declaretheorem[sibling=theorem]{corollary}
\declaretheorem[sibling=theorem]{proposition}
\declaretheorem[sibling=theorem,style=definition]{definition}
\declaretheorem[sibling=theorem,style=remark]{example}

\declaretheorem[sibling=theorem,style=remark]{remark}

\hypersetup{bookmarksdepth = 3} % Depth of sections/subs... to have bookmark links in the pdf;
\numberwithin{equation}{section}     % Makes labeled equations easier to find.

\setlist[enumerate,1]{label={\upshape(\alph*)},ref=\alph*}
\setlist[enumerate,2]{label={\upshape(\arabic*)},ref=\arabic*}

\newcommand{\M}{\mathcal{M}}
\newcommand{\R}{\mathbb{R}}

\newcommand{\N}{\mathbb{N}}

\newcommand{\cK}{\mathcal{K}}

\newcommand{\cS}{\mathcal{S}}

\newcommand{\Rnon}{\mathbb{R}_{\plus}}    % non-negative reals
\newcommand{\Rpos}{\mathbb{R}_{\plus\plus}} % positive reals
\newcommand{\Mat}[2][]{\ifthenelse{\equal{#1}{}}{\R^{{#2}\times{#2}}}{\R^{{#1}\times{#2}}}}
\newcommand{\Man}[2][]{\ifthenelse{\equal{#1}{}}{\Rnon^{{#2}\times{#2}}}{\Rnon^{{#1}\times{#2}}}}
\newcommand{\Map}[2][]{\ifthenelse{\equal{#1}{}}{\Rpos^{{#2}\times{#2}}}{\Rpos^{{#1}\times{#2}}}}
\renewcommand{\epsilon}{\varepsilon}
\renewcommand{\phi}{\varphi}

\begin{document}

\title{Time change for flows and Thermodynamic formalism}
\date{\today}


\begin{thanks}
{The authors would like to express their gratitude to Katrin Gelfert for interesting discussions and comments around the subject of this paper. Her ideas motivated part of this work. G.I.\ was partially supported  by Proyecto Fondecyt 1150058. G.I. and I.C. were partially supported by CONICYT PIA ACT172001 }
\end{thanks}

\author{Italo Cipriano} \address{Facultad de Matem\'aticas,
Pontificia Universidad Cat\'olica de Chile (PUC), Avenida Vicu\~na Mackenna 4860, Santiago, Chile}
\email{icipriano@gmail.com }
\urladdr{http://www.icipriano.org/}

\author{Godofredo Iommi} \address{Facultad de Matem\'aticas,
Pontificia Universidad Cat\'olica de Chile (PUC), Avenida Vicu\~na Mackenna 4860, Santiago, Chile}
\email{giommi@mat.puc.cl}
\urladdr{http://www.mat.puc.cl/\textasciitilde giommi/}

\begin{abstract}
This paper is devoted to study how do thermodynamic formalism quantities varies for time changes of suspension flows defined over countable Markov shifts.
We prove that in general no quantity is preserved. We also make a topological description of the space of suspension flows according to certain thermodynamic quantities. For example, we show that the set of suspension flows  defined over the full shift on a countable alphabet having finite entropy is open. Of independent interest might be a set of analytic tools we use to construct examples  with prescribed thermodynamic behaviour.
\end{abstract}

\maketitle

\section{Introduction}

A natural question that has attracted attention for quite some time is: How do ergodic properties of a flow varies with a time change? This question has been addressed under different regularity assumptions on the flow (and the time change)  and for a wide range of ergodic properties.  For example, under certain measurability assumptions it is known that ergodicity is preserved under time changes (see \cite[Theorem 5.1]{to} or \cite[Section 5]{par2}). On the other hand, mixing and weak-mixing are not necessarily preserved. Some examples have been extensively studied, for instance it was shown by Marcus \cite[Section 3]{ma} that under mild differentiability assumptions   time changes of the horocycle flow are topologically mixing. Spectral properties of the time change of  the horocylcle flow have been studied recently in \cite{fu,ti}. For more results along this lines see \cite{afu}. In the context of Axiom A flows, Parry \cite[Theorem 3]{par} showed that there exists a time change so that the Sinai-Ruelle-Bowen measure coincides with the measure of maximal entropy, a property he called synchronization. More recently, Gelfert and  Motter \cite{gm}, for a class of smooth flows, showed that several thermodynamic formalism quantities remained essentially unchanged by suitable time changes. The purpose of the present paper is to discuss how do thermodynamic quantities varies for time changes of suspension flows defined over countable Markov shifts. We stress that in this setting the phase space is no longer compact and this yields a very different thermodynamic behaviour.

The use of suspension flows over Markov shifts as a tool to study differentiable flows has a long story, probably starting with the study of the geodesic flow over the modular surface \cite{ku}. In the early 1970's Bowen \cite{bo1} and Ratner \cite{ra} constructed Markov partitions for Axiom A flows. These were later used by Bowen and Ruelle \cite{br} to study thermodynamic formalism for such flows. In this setting the symbolic model was a suspension flow defined over a sub-shift of finite type defined on a finite alphabet. Recently,  Lima and Sarig \cite{ls} have constructed  symbolic codings for three dimensional $C^{\infty}$ flows of positive entropy defined on compact manifolds. They showed that given $\chi >0$ there is a suspension flow defined over a finite entropy countable Markov shift that captures every hyperbolic measure for which its positive Lyapunov exponent is larger than $\chi$. Other examples of flows that can be coded by suspension flows over countable Markov shifts are some classes of geodesic flows defined over non-compact manifolds of variable pinched negative curvature \cite{dp, irv} and certain Teichm\"uler flows \cite{bg, h}.

In sub-section \ref{ss:nc} we provide a wide range of examples showing that, in strong contrast to the compact setting, no thermodynamic quantity is preserved in the context of suspension flows defined over countable Markov shifts. Most of these examples are constructed making use of analytic tools developed in the study of Dirichlet series (see sub-section \ref{sb:ex}). We also establish conditions on the time change that ensure that some thermodynamic quantities are actually preserved (Theorem \ref{thm:pre}). Finally, in Section \ref{sec:sta} we topologize the space of suspension flows defined over a fixed countable Markov shift. We study topological properties of the sets of flows having finite (or infinite) entropy, those which have (or not) measures of maximal entropy and other sets having particular ergodic properties.

\section{Countable Markov shifts }
This section is devoted, on the one hand, to recall the necessary definitions and results from thermodynamic formalism for countable Markov shifts that will be used in the article. On the other, more interestingly, we provide tools to construct examples of potentials for which the pressure has prescribed properties. The tools come from the study of zeta functions and Dirichlet series in number theory.

\subsection{Countable Markov shifts}
Let $T=(t_{ij})_{\N \times \N }$ be an infinite matrix 
of zeros and ones. Let
\begin{equation*}
\Sigma=\left\{ x \in \N ^{\N} : t_{x_{i} x_{i+1}}=1 \ \text{for every $i \in \N$}\right\}.
\end{equation*}
The \emph{shift map} $\sigma:\Sigma \to \Sigma$ is defined by $(\sigma(x))_i=x_{i+1}$.  The space $\Sigma$ is equipped with
the topology generated by the cylinder sets
\begin{equation*}
C_{a_1 \cdots a_n}= \{ x\in \Sigma: x_i=a_i \ \text{for
$i=1,\ldots,n$}\}.
\end{equation*}
In general it is a non-compact space. The pair $(\Sigma, \sigma)$ is called \emph{countable Markov shift}. In what follows we will always assume it to be  a transitive dynamical system.  Given a function $ \phi \colon \Sigma \to \R$ let
\[ V_{n}(\phi):= \sup \{| \phi(x)- \phi(y)| : x,y\in \Sigma, \ x_{i}=y_{i}
\ \text{for $i=1,\ldots,n$} \},
\]
where $x=(x_0, x_1, \ldots, )$ and $y=(y_0, y_1, \ldots, )$.  We say that $\phi$ has \emph{summable variation} if $\sum_{n=1}^{\infty} V_n(\phi)<\infty$.  We also say that
$\phi$ is \emph{locally H\"older} if there exist constants $K>0$ and
$\theta\in (0,1)$ such that $V_{n}( \phi) \le K \theta^{n}$ for all
$n\geq 1$. The following result summarizes work by Sarig  \cite{sa1},  which generalizes previous work by  Gurevich \cite{gu1, gu2} and by  Mauldin and Urba\'nski \cite{mu}.

\begin{theorem} \label{thm:pcms}
Let $(\Sigma, \sigma)$ be a transitive countable Markov shift and $\phi \colon \Sigma \to \R$  a function of summable variations. 
\begin{eqnarray*}
P_{\sigma}(\phi) &:=& \lim_{n \to
\infty} \frac{1}{n} \log \sum_{x:\sigma^{n}x=x} \exp \left(
\sum_{i=0}^{n-1} \phi(\sigma^{i}x)\right)  \chi_{C_{i_{0}}}(x) \\
&=& \sup \left\{ h(\nu) + \int \phi \ d \nu : \nu \in \M_{\sigma} \text{ and } - \int \phi \ d \nu < \infty \right\}\\
&=& \sup \{ P_{K}( \phi) : K\in \cK \}.
\end{eqnarray*}
The functional $P_{\sigma}$ is called \emph{Gurevich pressure} of $\phi$. The symbol $\chi_{C_{i_{0}}}(x)$ denotes the characteristic function of the cylinder $C_{i_{0}} \subset \Sigma$. The space of $\sigma-$invariant probability measures is denoted by $\M_{\sigma} $ and $h(\nu)$ denotes the entropy of the measure $\nu$. Finally, $\cK:= \{ K \subset \Sigma : K \ne \emptyset \text{ compact and }
\sigma\text{-invariant}\}$.
\end{theorem} 
It is possible to show that the limit in the definition of $P_{\sigma}$ always exists \cite{sa1}. Moreover, since
$(\Sigma,\sigma)$ is transitive, $P_{\sigma}(\phi)$ does not depend on $i_0$. A measure $\nu \in \M_{\sigma}$ such that $P_{\sigma}(\phi)= h(\nu) + \int \phi \ d \nu$ is called \emph{equilibrium measure} for $\phi$. Buzzi and Sarig \cite{busa} proved that a potential of summable variations has at most one equilibrium measure.

\subsection{The BIP case} \label{ss:bip}

We say that a countable Markov shift $(\Sigma, \sigma)$, defined by the transition matrix $T=(t_{ij})_{\N\times\N},$ satisfies the \emph{BIP condition} if and only if there exists a finite set $B \subset \N $ such that for every $a \in \N $ there exist $b,b'\in B$ with $t_{b a} t_{a b'}=1.$ For this class of countable Markov shifts, introduced by Sarig \cite{Sar01a}, the thermodynamic formalism  is similar to that of sub-shifts of finite type defined on finite alphabets.
 The following theorem\footnote{ The authors thank Thiago Costa Raszeja for pointing out that $s_{\infty}$ can be zero, by considering for example the full shift and $\varphi(x)=-x_0.$ } summarizes results proven by Sarig in \cite{sar3000} and by Mauldin and Urba\'{n}ski \cite{mu}.

\begin{theorem} \label{bip}
Let $(\Sigma, \sigma)$ be a transitive countable Markov shift satisfying the BIP condition and $\phi: \Sigma \to \R$ a non positive locally H\"older potential. Then, there exists $s_{\infty} \in [0, \infty]$ such that pressure function $t \to P_{\sigma}(t\phi)$ has the following properties
\begin{equation*}
P_{\sigma}(t \phi)=
\begin{cases}
\infty  & \text{ if  } t  < s_{\infty}; \\
\text{real analytic } & \text{ if  } t >s_{\infty}.
\end{cases}
\end{equation*}
Moreover, if $t> s_{\infty}$, there exists a unique equilibrium measure for $t \phi$.
\end{theorem}

We classify non positive  locally H\"older potentials according to the behaviour of the pressure function $t \mapsto P_{\sigma}(t \phi)$ at $t=s_{\infty}$.

\begin{definition}
Let $(\Sigma, \sigma)$ be a transitive countable Markov shift satisfying the BIP condition and $\phi: \Sigma \to \R$ a non positive  locally H\"older potential. We say that 
\begin{enumerate}
\item The potential $\phi$ is of \emph{infinite type} if for every $t \in \R$ we have $P_{\sigma}(t \phi)= \infty$.
\item  The potential $\phi$ is of \emph{continuous type} if there exists $t_0 \in \R$ such that $P_{\sigma}(t_0 \phi) <\infty$, $P_{\sigma}(s_{\infty} \phi)= \infty$ and $\lim_{t \to s_{\infty}^+}P_{\sigma}(t \phi)=\infty$.
\item  The potential $\phi$ is of \emph{discontinuous type} if there exists $t_0 \in \R$ such that $P_{\sigma}(t_0 \phi) <\infty$ and is not of continuous type, that is $P_{\sigma}(s_{\infty} \phi) <\infty$ or $\lim_{t \to s_{\infty}^+}P_{\sigma}(t \phi) <\infty$.
\end{enumerate}
\end{definition}

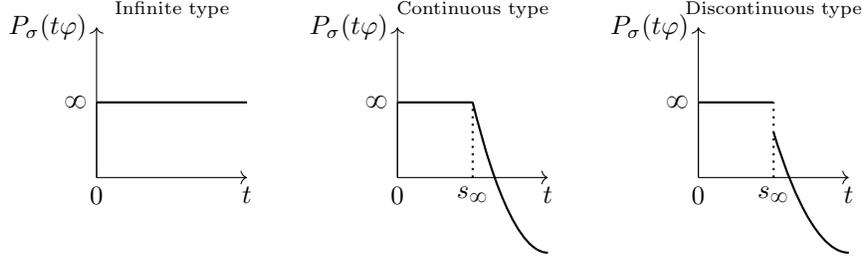
\begin{tikzpicture}

\draw[->] (0,0) -- (2,0) node[anchor=north] {$t$};
\draw	(0,0) node[anchor=north] {0};
\draw[->] (0,0) -- (0,2) node[anchor=east] {$P_{\sigma}(t \phi)$};
\draw[] (0,0) -- (0,1) node[anchor=east] {$\infty$};
\draw[thick] (0,1) -- (2,1);
\draw (1,2) node[anchor=south]{{\scriptsize Infinite type}};

\draw[->] (4+0,0) -- (4+2,0) node[anchor=north] {$t$};
\draw	(4+0,0) node[anchor=north] {0};
\draw[->] (4+0,0) -- (4+0,2) node[anchor=east] {$P_{\sigma}(t \phi)$};
\draw[] (4+0,0) -- (4+0,1) node[anchor=east] {$\infty$};
\draw[thick] (4+0,1) -- (4+1,1);
\draw (4+1,2) node[anchor=south]{{\scriptsize Continuous type}};
\draw[thick] (5,1) -- (5,1) parabola[bend at end] (6,-1);
\draw[] (5,0) -- (5,0) node[anchor=north] {$s_\infty$};
\draw[thick,dotted] (5,0) -- (5,1);

\draw[->] (8+0,0) -- (8+2,0) node[anchor=north] {$t$};
\draw	(8+0,0) node[anchor=north] {0};
\draw[->] (8+0,0) -- (8+0,2) node[anchor=east] {$P_{\sigma}(t \phi)$};
\draw[] (8+0,0) -- (8+0,1) node[anchor=east] {$\infty$};
\draw[thick] (8+0,1) -- (8+1,1);
\draw (8+1,2) node[anchor=south]{{\scriptsize Discontinuous type}};
\draw[thick] (9,0.6) -- (9,0.6) parabola[bend at end] (10,-1);
\draw[] (9,0) -- (9,0) node[anchor=north] {$s_\infty$};
\draw[thick,dotted] (9,0) -- (9,1);

\end{tikzpicture}

\subsection{The full-shift and analytic tools to construct examples.} \label{sb:ex}
Let $(\Sigma, \sigma)$ be the full-shift on countably many symbols, that is
\[ \Sigma:= \left\{ (x_n)_{n \in \N\cup \{0\}} : x_n \in \N \right\}.\]
It clearly satisfies the BIP property. If $\phi :\Sigma \to \R$ is a non positive locally constant potential, that is, for every $n\in\N$ we have that $\phi|C_n:= - \log \lambda_n$ for some $\lambda_n \in (1,\infty)$, then there is a simple formula for the pressure  (\cite[Example 1]{bi1})
\begin{equation} \label{Dirichlet}
e^{P_{\sigma}(t\phi )}= \sum_{n=1}^{\infty} \lambda_n^{-t}.
\end{equation}
Note that, in this case, $s_{\infty}=\inf\{t: \sum_{n=1}^{\infty} \lambda_n^{-t}<\infty\}.$
The above formulas allows for the use of analytic tools to construct examples of potentials having prescribed properties.  But not only that, we will relate the continuity type of the non positive potential $\phi$  with the nature of the singularities of the meromorphic extensions of the Dirichlet series \eqref{Dirichlet}. Indeed, we observe that a potential $\phi$ is of continuity type if (\ref{Dirichlet}) has meromorphic extension with maybe some poles, and $\phi$ is of discontinuity type if (\ref{Dirichlet}) has meromorphic extension with only branch points. We now recall some theorems on meromorphic extensions of certain Dirichlet series that we will use later. The following result can be found in \cite[Main Theorem I]{eie}.

\begin{theorem} \label{teo:e}
Let $P(x)=\prod_{j=1}^k(x+\delta_j)$ be a polynomial with real coefficients and $\delta_j \in \mathbb{C}$ satisfying  $\mathcal{R}e(\delta_j)>-1$ for $j=1, \ldots k.$ The associated Zeta function  $$ Z(s)=\sum_{n=1}^{\infty}\frac{1}{P(n)^s}$$ is holomorphic for $\mathcal{R}e(s)>\frac{1}{k}$ and it has an analytic continuation in the whole complex plane with only possible poles at $\frac{j}{k}$ for  $j=1,0,-1,-2,\ldots$ other than non-positive integers.
\end{theorem}

We recall the definitions of generalized Dirichlet series and asymptotic expansion. 

\begin{definition}
A generalized Dirichlet series is an infinite series 
\begin{equation}\label{genDirichletSeries}
L(s)=\sum_{n=1}^{\infty}a_n \lambda_n^{-s},
\end{equation}
where $s\in \mathbb{C}, \{a_n\}\subset \mathbb{C}, \{\lambda_n\}\subset \R^+$ and such that  $0<\lambda_1<\lambda_2<\cdots$ and $\lim_{n\to\infty}\lambda_n=\infty.$
\end{definition}

\begin{definition}
We say that  $f(t)$ has asymptotic expansion $$f(t)\sim \sum_{n=0}^{\infty} a_n f_n(t) \mbox{ as }t\to 0,$$ if $f(t)-\sum_{n=0}^{N-1} a_n f_n(t)$ is $O(f_N(t))$ as $t\to 0$ for any integer $N\geq 0.$
\end{definition}

The next result appears in \cite[p.313]{zieger}.
 
 \begin{theorem}\label{teo_zieger}
Let $L(s)$ be a generalized Dirichlet series as in (\ref{genDirichletSeries}). If the following conditions are satisfied:
\begin{enumerate}
\item the sequence $(\lambda_n)$ is growing at least as fast as some positive power of $n;$
\item the series $L(s)$ is convergent at some $s^* \in \mathbb{C};$ and
\item $f(t):=\sum_{n=1}^{\infty}a_n e^{-\lambda_n t}$ has asymptotic expansion
$$
f(t) \sim \sum_{n=-1}^{\infty} b_n t^n \mbox{ as $t\to 0.$}
$$
\end{enumerate}
Then $L(s)$ has a meromorphic extension to all $s\in\mathbb{C},$ with a simple pole of residue $b_{-1}$ at $s=1$ and no other singularities. Its values at non-positive integers are given by
$$
L(-n)=(-1)^n n! b_n \mbox{ }(n=0,1,2,\ldots).
$$
\end{theorem}
 
Note that the hypotheses of Theorems \ref{teo:e} and \ref{teo_zieger} are not disjoint, indeed there are maps that satisfy both, for example $L(s)=\sum_{n=1}^{\infty}\frac{1}{(n+q)^s}.$ This special case corresponds to the Hurwitz zeta function.

%\textbf{GI: la misma letra para hurwitz y las con log...cambiar y ajustar en las secciones siguientes} 

\begin{theorem}[Hurwitz zeta function]\label{Hurwitz_zeta}
The formally defined series $$\zeta(s,q)=\sum_{n=0}^{\infty}(n+q)^{-s}$$ with  $Re(s) > 1$ and $Re(q) > 0$ is absolutely convergent  and can be extended by analytic continuation to a meromorphic function on all $s\neq 1.$ At 
$s=1$ it has a simple pole with residue $1.$ Moreover, the function has an integral representation in terms of the Mellin transform as $$\zeta(s,q)=\frac{1}{\Gamma(s)} \int_0^\infty \frac{t^{s-1}e^{-qt}}{1-e^{-t}}dt$$ for $Re(s)>1$ and $Re(q)>0.$
\end{theorem}
 
\begin{remark}
We observe that if a Dirichlet series $L(s)$ with positive terms in the sum has holomorphic extension to $Re(s)>p,$ where $p$ is a simple pole, then $\lim_{s\to p^+} L(s)=+\infty.$
\end{remark}

We will construct and describe potentials of continuous type. Indeed, motivated on Theorem \ref{teo:e}, one can prove directly the following proposition.

\begin{proposition}
Let $P(x)$ be a real polynomial with degree $k$ such that $P>0$ in $(-1,\infty)$ and $P>1$ in $\N.$ If $\varphi$ is the locally constant potential defined by $\varphi |C_i= -\log P(i)$ for $i\in\mathbb{N},$ then  $s_{\infty}=\frac{1}{k}$ and $\phi$ is of continuous type.
\end{proposition}

Similarly, based on Theorems \ref{teo_zieger} and \ref{Hurwitz_zeta}, we obtain the following proposition.

\begin{proposition}
Let $(\lambda_n)_n$ be a strictly increasing sequence of real numbers with  $1\leq \lambda_1$, growing at least as fast as some positive power of $n$. Assume that there exists  $\epsilon>0$ and a sequence of real numbers $(a_n)_n$  such that the function $f:(0,\epsilon)\to\mathbb{R}$ defined by $f(t):=\sum_{m=1}^{\infty} e^{-\lambda_m  t}$ has asymptotic expansion
$$
f(t)\sim \sum_{n=-1}^{\infty} a_n t^n \mbox{ as $t\to 0.$}
$$ 
If $\varphi$ is the locally constant potential defined by $\varphi |C_n:= -\log  \lambda_n$ for $n\in\mathbb{N},$ then
$s_{\infty}=1$ and $\varphi$ is of continuous type.
\end{proposition}

We now state a result on the meromorphic extension of certain Dirichlet series \cite[Theorem 1]{Grabner_Thuswaldner96} that is related to the construction of potentials of discontinuous type.

\begin{theorem} \label{thm:log}
Let $\eta$ and $\theta$ be real numbers, then the Dirichlet series
$$
\mathcal{L}_{\eta,\theta}(s)=\sum_{k=2}^{\infty}\frac{(\log k)^{\eta}}{(k(\log k)^{\theta})^s}
$$
admits an analytic continuation to the whole complex plane except at the line joining $1$ with $-\infty.$ This line gives a branch cut of the function, whose nature depends on the parameters. The singular expansion of the function around $s=1$ starts with 
$$
\begin{aligned}
\Gamma(\eta-\theta+1)(s-1)^{\theta-\eta-1} & \mbox{ for }\theta-\eta\notin \mathbb{N}\\
\frac{(-1)^{m-1}}{(m-1)!}(s-1)^{m-1}\log \frac{1}{s-1} & \mbox{ for }\theta-\eta=m \in \mathbb{N}.
\end{aligned}
$$
\end{theorem}

\begin{remark} \label{rem:tilde}
We observe that if a Dirichlet series $L(s)$ with positive terms in the sum admits an analytic continuation to the whole complex plane except the line joining $1$ with $-\infty,$ then the map $[1,\infty) \ni s \to \tilde{L}(s)\in\mathbb{R}$ is continuous, where 
$$
\tilde{L}(s)=
\begin{cases}
L(s) & \mbox{ if }s>1;\\
\lim_{s\to 1^+} L(s) & \mbox{ if }s=1.\\
\end{cases}
$$
\end{remark}

From Theorem \ref{thm:log} and Remark \ref{rem:tilde}, we obtain the following proposition that allow us to construct and describe potentials of discontinuous type.

\begin{proposition}
Let $\theta >0.$ If $\varphi$ is the locally constant potential defined by $\varphi |C_n:= -\log \left((n+1)\log^{\theta} (n+1)\right)$ for $n\in\mathbb{N},$ then $s_{\infty}=1$ and $\varphi$ is of discontinuous type.
\end{proposition}

\section{Suspension flows over countable Markov shifts}
In this section we recall definitions and properties of the thermodynamic formalism for suspension flows over countable Markov shifts.

\subsection{Suspension semi-flows and invariant measures}
Let $(\Sigma, \sigma)$ be a countable Markov shift and let $\tau \colon \Sigma \to \R$ be a positive continuous function such that for every  $x\in\Sigma$ we have $\sum_{i=0}^{\infty}\tau(\sigma^i x)=\infty.$ Consider the space
\begin{equation*}\label{eq:flow phase }
Y= \{ (x,t)\in \Sigma  \times \R \colon 0 \le t \le\tau(x)\}
\end{equation*}
with the points $(x,\tau(x))$ and  $(\sigma(x),0)$ identified for
each $x\in \Sigma $. The \emph{suspension semi-flow} over $\sigma$
with \emph{roof function} $\tau$ is the semi-flow $\Phi = (
\varphi_t)_{t \ge 0}$ on $Y$ defined by
\begin{equation*}
 \varphi_t(x,s)= (x,
s+t) \ \text{whenever $s+t\in[0,\tau(x)]$.}
\end{equation*}
A probability measure $\mu$  on $Y$  is
\emph{$\Phi$-invariant} if $\mu(\varphi_t^{-1}A)= \mu(A)$ for every
$t \ge 0$ and every measurable set $A \subset Y$. Denote by $\M_\Phi$ the space of $\Phi$-invariant probability
measures on $Y.$
The space $\M_\Phi$ is closely related to the space $\M_\sigma$ of $\sigma$-invariant probability measures on $\Sigma $. Let
\begin{equation*} 
\M_\sigma(\tau):= \left\{ \mu \in \mathcal{M}_{\sigma}: \int \tau \text{d} \mu < \infty \right\}.
\end{equation*}
Note that if $\Sigma$ is compact then $\M_\sigma(\tau)=\M_{\sigma}$. A result by Ambrose and Kakutani \cite{ak} implies that if $m$ denotes the Lebesgue measure then
\begin{equation} \label{eq:meas}
\frac{(\mu \times m)|_{Y} }{(\mu \times m)(Y)} \in \M_{\Phi}.
\end{equation}
Moreover,  if  $(\Sigma , \sigma)$ is a sub-shift of finite type defined over a finite alphabet,  then equation \eqref{eq:meas} defines a bijection between $:\M_\sigma$ and 
$\M_{\Phi}$. When  $(\Sigma , \sigma)$ is a countable Markov shift and $\tau:\Sigma \to \R^+$ is  not  bounded above then there is a bijection between $\M_\sigma(\tau)$ and $M_{\Phi}$. Note though that there might be measure $\nu \in \M_{\sigma}$ such that
$\int \tau \ d \nu  = \infty$. In this case the measure obtained in  equation  \eqref{eq:meas} is an infinite flow invariant measure. The more subtle case is when $(\Sigma , \sigma)$ is a countable Markov shift and $\tau:\Sigma \to \R^+$ is  not  bounded away from zero, since it is possible that for an infinite (sigma-finite) $\sigma-$invariant measure $\nu$  we have $\int  \tau \ d \nu <\infty$.  In this case the measure $(\nu \times m)|_{Y} /(\nu \times m)(Y) \in \M_\Phi$.

\subsection{Thermodynamic formalism}

The entropy of a flow with respect to an invariant measure,  denoted $h_\Phi(\mu)$, can be defined as  the entropy of the corresponding time one map. The entropy of the flow is related to the entropy of the shift. The following formula was obtained by Abramov \cite{a} and later generalised by Savchenko \cite[Theorem 1]{sav}.  Let $\mu \in \mathcal{M}_\Phi$ be an ergodic measure  such that  $\mu=(\nu \times m)|_{Y} /(\nu \times m)(Y)$, where $\nu$ is a sigma-finite (finite or infinite) invariant measure for the shift with $\int \tau \ d\nu < \infty$. Then
\begin{equation}
h_{\Phi}(\mu)=\frac{h_{\sigma}(\nu)}{\int \tau  \text{d} \nu}.
\end{equation}
It also possible to relate the integral of a potential on the flow to a corresponding one on the base. Indeed, given a continuous function $g \colon Y\to\R$ we define the function
$\Delta_g\colon\Sigma\to\R$~by
\[
\Delta_g(x)=\int_{0}^{\tau(x)} g(x,t) \, ~{\rm d}t.
\]
The function $\Delta_g$ is also continuous, moreover if $\mu \in \M_{\Phi}$ is the normalisation of  $\nu \times m$ then
\begin{equation*} \label{eq:rela}
\int_{Y} g \, ~{\rm d} \mu = \frac{\int_\Sigma \Delta_g\, ~{\rm d}
\nu}{\int_\Sigma\tau \, ~{\rm d} \nu}.
\end{equation*}
Thermodynamic formalism for suspension flows over countable Markov shifts has been studied by several people (see for example \cite{bi1, ij, ijt, jkl, ke, sav}). The following result summarizes some of the results that have been obtained on  thermodynamic formalism for suspension flows over countable Markov shifts.

\begin{theorem} \label{pre}
Let $(\Sigma, \sigma)$ a topologically mixing countable Markov shift and $\tau:\Sigma \to \R^+$ a roof function of summable variations.  Let $(Y, \Phi)$ be the associated suspension semi-flow. Let $g:Y \to \R$ be a function such that $\Delta_g:\Sigma \to \R$ is locally H\"older. Then the following equalities hold
\begin{eqnarray*}
P_{\Phi}(g)&:=&\lim_{t \to \infty} \frac{1}{t} \log \left(\sum_{\phi_s(x,0)=(x,0), 0<s \leq t} \exp\left( \int_0^s g(\phi_k(x,0)) \text{d}k \right) \chi_{C_{i_0}}(x) \right) \\
&=& \inf\{t \in \R : P_{\sigma} (\Delta_g - t \tau) \leq 0\} =\sup \{t \in \R : P_{\sigma} (\Delta_g - t \tau) \geq 0\} \\
&=& \sup \{ P_{\sigma|K}( \phi) : K\in \cK \},\\
&=& \sup \left\{ h_{\mu}(\Phi) +\int_Y g \text{d} \mu : \mu\in
\mathcal{E}_{\Phi} \text{ and } -\int_Y g \, \text{d}\mu <\infty \right\},\end{eqnarray*}
where $\cK$ is the set of all compact and $\Phi-$invariant sets and  $P_K$ is the
classical topological pressure of the potential $\phi$ restricted to the
compact and $\sigma$-invariant set $K$ and  $\mathcal{E}_\Phi $ is the set of ergodic $\Phi-$invariant measures.
\end{theorem}

\begin{corollary} \label{entropia}
Let $(\Sigma, \sigma)$ a topologically mixing countable Markov shift and $\tau:\Sigma \to \R^+$ a roof function of summable variations.  Let $(Y, \Phi)$ be the associated suspension flow. The topological entropy of the semi-flow is given by
\begin{equation*}
h(\Phi)=\inf\{t \in \R : P_{\sigma} (- t \tau) \leq 0\}.
\end{equation*}
\end{corollary}

Recall that a measure $\mu \in \mathcal{E}_{\Phi}$ is called an \emph{equilibrium measure} for $g$ if $P_\Phi(g)= h(\mu) + \int g \ d \mu$.  
The following result is a summary of  \cite[Theorem 3.4 and  3.5]{ijt},

\begin{theorem} \label{thm_eqstate}
Let $\Phi$ be a finite entropy suspension semi-flow on $Y$ defined over a countable Markov
shift  $(\Sigma, \sigma)$  and a locally H\"older roof function $\tau$. Let $g \colon Y \to \R$ be a continuous function such
that $\Delta_g$ is locally H\"older. In the following cases there exists an equilibrium measure for $g$;
\begin{enumerate}
\item If $P_\sigma(\Delta_g -P_{\Phi}(g) \tau)=0$ and $\Delta_g -P_{\Phi}(g) \tau$ has an equilibrium measure $\nu_g$ satisfying $\int \tau ~{\rm d} \nu_g < \infty$;
\item   If $P_\sigma(\Delta_g -P_{\Phi}(g) \tau)=0$ and the potential $\Delta_g -P_{\Phi}(g) \tau$ has an infinite Ruelle-Perron-Frobenius measure $\nu_g$ (i.e. $\nu_g=hm$ where $h$ and $m$ are the density and the conformal measure provided by the Ruelle-Perron-Frobenius Theorem with pontential $g$)  and  $\int \tau ~{\rm d}\nu_g < \infty$.
   \end{enumerate}
In any other case the potential $g$ does not have an equilibrium measure. Moreover, every potential $g$, for which $\Delta_g$ is locally H\"older, has at most one equilibrium state.
\end{theorem}

A throughout account of the case in which $\Sigma$ is a compact sub-shift of finite type can be found in \cite{pp}. 

\section{Time change for flows.}
We begin this section showing that two suspension flows defined over the same base system are strongly related. Not only one is orbit equivalent to the other, but one is a time change of the other. Indeed, let $\Sigma$ be a fixed transitive  countable Markov shift and let $\tau_1: \Sigma \to \R$ and $\tau_2: \Sigma \to \R$ be two positive roof functions. Denote by $(Y_1, \Phi_1)$ and $(Y_2, \Phi_2)$ the corresponding semi-flows

\begin{lemma}
The semi-flow $(Y_1, \Phi_1)$ is a time change of  $(Y_2, \Phi_2)$.
\end{lemma}

\begin{proof}
The map $\pi:Y_1 \to Y_2$ defined by
\begin{equation*}
\pi(x,s)=\left(x ,  \frac{\tau_2(x)}{\tau_1(x)} s  \right),
\end{equation*}
preserves the orbit structure. It actually send leafs to leafs and corresponds to the time change. 
\end{proof}
It should be stressed though that the spaces $Y_1$ and  $Y_2$ are different. With this time change map we can relate potentials defined in $Y_1$ with those defined in $Y_2$.
Let $\psi_2: Y_2 \to \R$ and define  $\psi_1: Y_1 \to \R$ by $\psi_1(x,r):= \psi_2 \circ \pi (x,r)$. Note that
\begin{equation*}
\Delta_{\psi_2}(x)= \int_0^{\tau_2(x)} \psi_2(x,r) \ dr \text{ and }
\Delta_{\psi_1}(x)= \int_0^{\tau_1(x)} \psi_2\left(x,\frac{\tau_2(x)}{\tau_1(x)}r  \right) \ dr.
\end{equation*}

\begin{remark}
Note that both functions are equal, that is $\Delta_{\psi_2}(x)=\Delta_{\psi_1}(x)$. For simplicity we will denote it by $\Delta_{\psi}$.
\end{remark}

\subsection{The compact setting: preservation of thermodynamic quantities}
This sub-section is devoted to prove that when the space is compact, thermodynamic quantities are preserved by regular time changes. Indeed, we prove that the pressure changes at most by a factor that depends on the  quotient of the roof functions. The closer this quotient is to one the less the pressure changes. Note that  thermodynamic formalism for Axiom A flows was studied by Bowen and Ruelle \cite{br} making use of the fact that these flows can be modeled  by suspension flows over sub-shifts of finite type with a H\"older roof function. Thus, the following result not only shows that thermodynamic quantities are preserved by regular time changes at a symbolic level but  also in  the differentiable category of Axiom A flows.

\begin{theorem} \label{thm: AxiomA}
Let $(\Sigma, \sigma)$ be a transitive sub-shift of finite type defined over a finite alphabet and let $\tau_1: \Sigma \to \R$ and $\tau_2: \Sigma \to \R$ be two positive H\"older roof functions. Denote by $(Y_1, \Phi_1)$ and $(Y_2, \Phi_2)$ the corresponding semi-flows. Let $\psi_2: Y_2 \to \R$ be a potential such that $\Delta_{\psi_2}$ is a H\"older function and define  $\psi_1(x,r):= \psi_2 \circ \pi (x,r)$. Then there exists a constant $C>0$ such that
\begin{equation*}
 \frac{P_{\Phi_1}(\psi_1)}{C} \leq P_{\Phi_2}(\psi_2) \leq CP_{\Phi_1}(\psi_1).
 \end{equation*}
Moreover, both $\psi_1$ and $\psi_2$ have a unique equilibrium state.
\end{theorem}

\begin{proof}
Since both $\tau_1$ and $\tau_2$ are positive continuous functions defined over a compact space, there exists $C>0$ such that for every $x \in \Sigma$ we have
\begin{equation*} \label{sim}
\frac{1}{C} < \frac{\tau_2(x)}{\tau_1(x)} < C.
\end{equation*}
Therefore  for every positive $t \in \R$ we have that $-\frac{t}{C}\tau_1 \geq -t \tau_2 \geq -tC \tau_1.$
For any $t \geq P_{\Phi_2}(\psi_2)$ we have that $0 \geq P_{\sigma}(\Delta_{\psi_2} - t \tau_2) \geq P_{\sigma}( \Delta_{\psi_1}-tC \tau_1). $
Therefore $ P_{\Phi_1}(\psi_1) \leq C P_{\Phi_2}(\psi_2).  $
On the other hand, if $t < P_{\sigma}(\psi_2)$ then
\[0 \leq  P_{\sigma}(\Delta_{\psi_2} - t \tau_2) \leq  P_{\sigma}\left( \Delta_{\psi_1}-\frac{t}{C}\tau_1 \right),  \]
thus  for every $t < P_{\Phi_2}(\psi_2)$ we have that  $t/C \leq P_{\Phi_1}(\psi_1)$. Therefore
\[ \frac{P_{\Phi_2}(\psi_2)}{C} \leq P_{\Phi_1}(\psi_1). \]
An analogous argument holds if $t <0$. The fact that both systems have unique equilibrium measures for H\"older potentials was proven by Bowen and Ruelle in \cite{br}.
\end{proof}

\subsection{The non compact setting: non-preservation of thermodynamic quantities} \label{ss:nc}

 This sub-section is devoted to show that if the base of the suspension flow is not assumed to be compact then, despite  the regularity assumed in the roof function or in the potentials considered, no natural thermodynamic quantity is preserved by time changes. We will exhibit explicit  examples showing this. In order to construct such examples we will mostly make use of the techniques developed in sub-section \ref{sb:ex}. Indeed, we have the following results.

%\textbf{GI: arreglar la notacion de las funcion es de dirichlet}

\begin{lemma} \label{lem_c1}
Let $(\Sigma, \sigma)$ be the full-shift on a countable alphabet. Let  $a,b>0$ and  $\varphi_{(a,b)}$ the family of locally constant potentials defined, for every $n \in \N$,  by $\varphi_{(a,b)} | C_n := -b\log (n+a)$. Denote  by $(Y, \Phi_{(a,b)})$ the suspension semi flow with base $\Sigma$ and roof function $\tau=-\varphi_{(a,b)}.$  This family of flows has the following properties.
\begin{enumerate}
\item \label{lemma1_one} $s_{\infty}(\Phi_{(a,b)})=\frac{1}{b}$ independently of $a>0.$
\item \label{lemma1_two} $P_{\sigma}(\varphi_{(a,b)})= P_{\sigma}(b\varphi_{(a,1)}),$ in particular $h( \Phi_{(a,b)} )=\frac{h( \Phi_{(a,1)})}{b}.$ 
\item \label{lemma1_three} For every $b,c>0,$ there exists $a> 0$ such that $h( \Phi_{(a,b)} )>c.$
\item \label{lemma1_four} For every $c>1,$ there exists a unique $a>0$ such that $h( \Phi_{(a,1)} )=c.$
\item \label{lemma1_five} For every $0<c<1<d,$ there exists a unique pair $(a,b)$ such that $s_{\infty}(\Phi_{(a,b)})=c$ and $h( \Phi_{(a,b)} )=d.$
\end{enumerate}
\end{lemma}

\begin{proof}
Part (\ref{lemma1_one}) follows from the observation that $e^{P_{\sigma}(\varphi_{(a,b)})}=\zeta(b,a+1)$ and an application of  Theorem \ref{Hurwitz_zeta}. Part (\ref{lemma1_two}) follows from definition, indeed:
$$
\begin{aligned}
P_{\sigma}(b\varphi_{(a,1)})&= \log \sum_{n=1}^{\infty} \exp\left(b\varphi_{(a,1)}| C_n\right)= \log \sum_{n=1}^{\infty} \exp\left(-b\log (n+a) \right)\\
&= \log \sum_{n=1}^{\infty} \exp\left(\varphi_{(a,b)}| C_n \right)=P_{\sigma}(\varphi_{(a,b)}).
\end{aligned}
$$
For part (\ref{lemma1_three}), it is sufficient to prove that for every $c>1$ there exists $a>0$ such that $h( \Phi_{(a,1)} )>c,$ which is equivalent to $$\inf\{t\in\mathbb{R}: P_{\sigma}( t \varphi_{(a,1)} )\leq 0\}>c.$$ 
We will prove that there exists $a>0$ such that for $e^{P_{\sigma}(t \varphi_{(a,1)})}=e^{P_{\sigma}(\varphi_{(a,t)})}=\zeta(t,a+1)$ we have that 
$$
 \zeta(t,a+1) >1 \mbox{ for all }t<c.
$$
For any fixed $a>0,$ we have that $\zeta(t,a+1)$ is decreasing in $t>1,$ therefore, it is enough to prove that there exists $a>0$ such that
\begin{equation*}
 \zeta(c,a+1) >1.
\end{equation*}
We notice that for every $c>1,$ $\zeta(c,1)>1,$ this together with the fact that the map $a\mapsto \zeta(c,a+1)$ is continuous, proves the result. 

For  part (\ref{lemma1_four}), let $a^*>0$ such that $h( \Phi_{(a^*,1 )})>c.$ By \cite[ Theorem 2.1]{alzer2015}, the map $a\mapsto \zeta(c,1+a)$ is strictly decreasing in $(0,\infty),$ and its clear that by taking $a$ big enough $ \zeta(c,1+a)<1,$ therefore, there exists a unique $a^{**}>a^*$ such that $\zeta(c,1+a^{**})=1.$ This implies that  $h( \Phi_{(a^{**},1)})=c,$ which completes the proof.

For part (\ref{lemma1_five}), let $a>0$ such that $h( \Phi_{(a,1)})=\frac{d}{c}.$ Then, $h(\Phi_{(a,\frac{1}{c})})=d$ and $s_{\infty}(\Phi_{(a,\frac{1}{c})})=c.$
\end{proof}

We state the following elemental lemma from calculus in order to justify the definition of certain potentials.

\begin{lemma}\label{lem:calc} Let $A(n,\theta):=\frac{1}{(\theta-1)\log^{\theta-1}(n)}$ and $B(n,\gamma):=\sup\{\theta>1: \gamma A(n,\theta) >1\}$ for $n\in\N,\theta>1$ and $\gamma>0.$ 
\begin{enumerate}
\item For every $r=2,3,\ldots$
$$
A(r,\theta)<\sum_{k=r}^{\infty} \frac{1}{k\log^{\theta} k}<A(r-1,\theta).
$$
\item For every $\epsilon>0,$ 
$$4=\min\left\{r\in \{2,3,\ldots\}: A(r-1,\theta)-A(r,\theta) \mbox{ is uniformly bounded for }\theta\in(\epsilon,\infty) \right\}.$$
\item For every $r\geq 4,$ $\gamma>0,$ $B(r,\gamma)<B(r-1,\gamma),$ the map $(1,\infty)\ni\theta\mapsto \sum_{k=r}^{\infty} \frac{\gamma}{k\log^{\theta} k}$ is continuous strictly decreasing with 
$$
\sum_{k=r}^{\infty} \frac{\gamma}{k\log^{\theta} k}=
\begin{cases}
<1 \mbox{ if } \theta >B(r-1,\gamma);\\
>1 \mbox{ if } \theta <B(r,\gamma).
\end{cases}
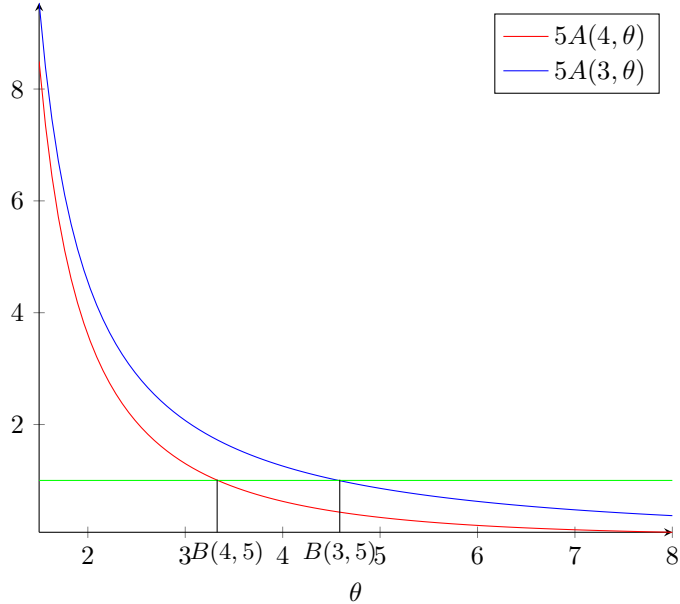
$$
\end{enumerate}

\end{lemma}

\begin{tikzpicture}
\begin{axis}[
    axis lines = left,
    xlabel = $\theta$,
    ylabel = {$$},
]
\addplot [
    domain=1.5:8, 
    samples=100, 
    color=red,
]
{5/((x - 1)*ln(4)^(x - 1))};
\addlegendentry{$5A(4,\theta)$}

\addplot [
    domain=1.5:8, 
    samples=100, 
    color=blue,
    ]
    {5/((x - 1)*ln(3)^(x - 1))};
\addlegendentry{$5A(3,\theta)$}

\addplot [
    domain=1.5:8, 
    samples=100, 
    color=green,
    ]
    {1};

\end{axis}

 \draw (2.37,0) coordinate (a_1) -- (2.37,0.7) coordinate (a_2);
 \draw (4,0) coordinate (b_1) -- (4,0.7) coordinate (b_2);

\draw  (2.5,0) node (yaxis) [below] {{\small $B(4,5)$}};
\draw  (4,0) node (yaxis) [below] {{\small $B(3,5)$}};       

\end{tikzpicture}

\begin{lemma}\label{Lem_disc} 
Let $(\Sigma, \sigma)$ be the full-shift on a countable alphabet. Let $\theta\in \mathbb{R}^+, k\in\N, 0<\gamma\leq 2+k$ and  $\varphi_{\theta,k,\gamma}$ the family of locally constant potentials defined, for every $n \in \N$,  by 
$\varphi_{\theta,k,\gamma} | C_n := -\log \left( \frac{(n+k+1)\log^{\theta} (n+k+1)}{\gamma} \right).$ If $\varphi_{\theta,k,\gamma}<0,$ denote by $(Y, \Phi_{\theta,k,\gamma})$ the suspension semi flow with base $\Sigma$ and roof function $\tau=-\varphi_{\theta,k,\gamma}.$  This family of flows has the following properties.
\begin{enumerate}
\item \label{lem11_1} For every $\theta\in\mathbb{R}^+$ and $k\in\N,$  $\varphi_{\theta,k,\gamma} <0$ and  $s_{\infty}(\Phi_{\theta,k,\gamma})=1.$
\item \label{lem11_2} The pressure satisfies
$$
P_{\sigma}(\varphi_{\theta,k,\gamma})=
\begin{cases}
<\infty & \mbox{ if }\theta>1;\\
=\infty & \mbox{ if }\theta\leq 1.
\end{cases}
$$
\item \label{lem11_3} If $\theta>1$ and $k\geq 2,$  $P\left(\varphi_{\theta_{k,\gamma},k,\gamma}\right)=0$ for a unique $\theta_{k,\gamma}\in (B(k+2,\gamma),B(k+1,\gamma)).$ Moreover, 
$$
P\left(\varphi_{\theta,k,\gamma}\right)=
\begin{cases}
\mbox{ positive } & \mbox{ if } \theta <\theta_{k,\gamma};\\
\mbox{ negative } & \mbox{ if } \theta >\theta_{k,\gamma},
\end{cases}
$$
and
$$
h_{top}\left(\Phi_{\theta,k,\gamma}\right)=
\begin{cases}
>1 & \mbox{ if } \theta <\theta_{k,\gamma};\\
=1 & \mbox{ if } \theta \geq \theta_{k,\gamma}.
\end{cases}
$$
\item \label{lem11_4} If $\theta>1,$ the right-derivative of the pressure satisfies
$$
\frac{d}{ds}P_{\sigma}(s\varphi_{\theta,k,\gamma})|_{s=1}=
\begin{cases}
>-\infty & \mbox{ if }\theta>2;\\
=-\infty & \mbox{ if } 1<\theta\leq 2.
\end{cases}
$$
\end{enumerate}
\end{lemma}

\begin{proof}

To prove (\ref{lem11_1}), we notice that for every $\theta>0, k\in\N, 0<\gamma\leq 2+k$ $$\varphi_{\theta,k,\gamma}<-\log \left( \frac{(2+k) \log^{\theta} (2+k)}{\gamma} \right)<-\log ( \log^{\theta} (2+k))<0,$$  and 
\begin{equation}\label{lem11_ident}
e^{ P_{\sigma}(s\varphi_{\theta,k,\gamma}) }=\gamma \mathcal{L}_{0,\theta}(s)-\sum_{n=2}^{k+1}\frac{\gamma}{(k(\log k)^{\theta})^s}=\sum_{n=k+2}^{\infty}\frac{\gamma}{(k(\log k)^{\theta})^s},
\end{equation}
in particular, an immediate consequence of Theorem \ref{thm:log} is $s_{\infty}(\Phi_{\theta,k,\gamma})=1.$ The proof of (\ref{lem11_2}) follows from (\ref{lem11_ident}) and the equivalences $P_{\sigma}(s\varphi_{\theta,k,\gamma})=\infty$ iff $\mathcal{L}_{0,\theta}(s)=\infty$ iff $s<1$ or $s=1$ and $\theta\leq 1.$ The proof of (\ref{lem11_3}) follows from Lemma \ref{lem:calc}. Finally, the proof of (\ref{lem11_4}) follows from right-differentiating in (\ref{lem11_ident}),
indeed, $$
\begin{aligned} 
\frac{d}{ds} P_{\sigma}(s\varphi_{\theta,k,\gamma})|_{s=1}&=\gamma \left(\frac{d}{ds} \mathcal{L}_{0,\theta}(s)|_{s=1}-\frac{d}{ds} \sum_{n=2}^{k+1}\frac{1}{(k(\log k)^{\theta})^s}|_{s=1} \right)e^{ -P_{\sigma}(s\varphi_{\theta,k,\gamma})}\\
&=  \left(\frac{d}{ds} \mathcal{L}_{0,\theta}(s)|_{s=1}+O(1) \right) O(1),
\end{aligned}
$$
and 
$$
\frac{d}{ds}\mathcal{L}_{0,\theta}(s)|_{s=1}=
\begin{cases}
>-\infty & \mbox{ if }\theta>2;\\
=-\infty & \mbox{ if } 1<\theta\leq 2,
\end{cases}
$$
which concludes the proof.

\end{proof}

Motivated by Lemmas \ref{lem_c1} and \ref{Lem_disc} we construct examples showing that no natural thermodynamic quantity is preserved by time changes for suspension semi-flows with base the full shift on $\N$.\\

Let $\Sigma$ be the full shift on $\N.$ We will define in each example particular roof functions $\tau_1,\tau_2:\Sigma \to \R^+$ constant on cylinder of length one and we will denote by $(Y_1, \Phi_1)$ and $(Y_2, \Phi_2)$ the suspension semi-flows with base $\Sigma$ and roof functions $\tau_1$ and $\tau_2,$ respectively.
%(existence of a unique maximal entropy measure, finite entropy flow, existence of equilibrium measure, finite equilibrium measure, analytic entropy map)
\begin{example}[\emph{Time change of a flow with a unique maximal entropy measure such that the time changed flow does not have a measure of maximal entropy.}] \label{ex:1}

Consider the flow $\Phi_1$ associated to $\tau_1|C_n:= \log (n(n+1))$. Choose $k \in \N$ such that
\[  \sum_{n=1}^{\infty} \frac{1}{(n+k) \log ^2(n+k)} <1\]
and consider the flow $\Phi_2$ associated to $\tau_2|C_n:=\log ((n+k) \log(n+k))$. Note that
\begin{equation*}
P_{\sigma}(-t\tau_1)= \log \sum_{n=1}^{\infty} \left( \frac{1}{n(n+1)} \right)^t.
\end{equation*}
In particular we  have that the entropy of the flow is (see Corollary \ref{entropia})
\begin{equation*}
h(\Phi_1):= \inf \left\{ t \in \R : P_{\sigma}(-t\tau_1) \leq 0 \right\}=1.
\end{equation*}
Moreover, the normalization of the  measure $\mu_1= \nu_1 \times m$, where $\nu_1$ is the equilibrium measure for $-\tau_1$ and $m$ is the Lebesgue measure, is the measure of maximal entropy for $\Phi_1$. On the other hand,
\begin{equation*}
P_{\sigma}(-t\tau_2)= \log  \sum_{n=1}^{\infty} \left( \frac{1}{(n+k) \log ^2 (n+k)}\right)^t.
\end{equation*}
Thus,
\begin{equation*}
P_{\sigma}(-t \tau_2)=
\begin{cases}
\infty & \text{ if } t <1;\\
\text{negative} & \text{ if } t \geq1.
\end{cases}
\end{equation*}
Therefore, $h(\Phi_2):= \inf \left\{ t \in \R : P_{\sigma}(-t\tau_2) \leq 0 \right\}=1$. But  there is no measure of maximal entropy in virtue of Theorem \ref{thm_eqstate}.
 \end{example}

\begin{example}[\emph{Time change of a finite entropy flow such that the time changed flow has infinite entropy.}]
 The flow $\Phi_1$ associated to $\tau_1 |_{C_n}:= \log (n(n+1))$ has entropy equal to one (see Example \ref{ex:1}), while, the flow $\Phi_2$ associated to $\tau_2|C_n:=1$ has infinite entropy.
\end{example}

\begin{example}[\emph{Time changed flows such that there is no bijection between the spaces of invariant probability measures.}]

Consider the flow $\Phi_1$ associated to $\tau_1 |_{C_n}:= \log ((n(n+1))$. As we saw in Example \ref{ex:1} this flow has entropy equal to one and the measure of maximal entropy is the normalization of the  product measure $\mu_1= \nu_1 \times m$, where $\nu_1$ is the Gibbs (Bernoulli) measure given by
\[ \nu_1(C_n)= \frac{1}{n(n+1)}.\]
Consider now the flow $\Phi_2$ associated to $\tau_2 |C_n:=n.$ Note that
\[P_{\sigma}(-t \tau_2)= \log \sum_{n=1}^{\infty} \left( e^{-t n} \right).\]
We claim that the normalized (with respect to $Y_2$) measure $\nu_1 \times m \notin \M_{\Phi_2}$. Indeed, in order for $(\nu_1 \times m)/ (\nu_1 \times m (Y_2)) \in \M_{\Phi_2}$
a necessary condition that needs to be satisfied is that
\[\int \tau_2 \ d \nu_1< \infty.\]
However,
\begin{equation*}
\int \tau_2 \ d \nu_1= \sum_{n=1}^{\infty} \frac{n}{n(n+1)} = \infty.
\end{equation*}
In particular, the measure $\nu_1 \times m$ is an infinite $\Phi_2-$invariant measure.
\end{example}

In the next two examples we use a tool mentioned in \cite[Section 4.3]{bi1}: Given a suspension semiflow $\Phi$ on $Y$ over $\Sigma,$ we can chose any locally H\"older function $\varphi:\Sigma\to \R,$ and by a result in \cite{brw}, there exists a continuous function $\psi: Y\to \R$ such that $\Delta_{\psi}=\varphi.$ In particular, we can implicitly define a continuous function $\psi: Y\to \R$ by defining $\triangle_{\psi}|C_n:=a_n,$ where $(a_n)_n$ is a sequence of real numbers.

\begin{example}[\emph{Time changed flows such that a potential has an equilibrium measure in one flow and not in the other.}] \label{ex:eqno}
Define $\Delta_{\psi}|C_n := -2\log \log (n+5).$ Consider the flow $\Phi_1$ associated to
$\tau_1|C_n:=2\log (n+5)$. In this case
\begin{equation*}
P_{\sigma}(\Delta_{\psi} -t \tau_1)= \log \sum_{n=1}^{\infty} \frac{1}{(n+5)^{2t} \log ^2 (n+5)}.
\end{equation*}
Therefore
$P_{\Phi}(\psi)=1/2$ and $P_{\sigma}(\Delta_{\psi} -1/2 \tau_1) <0$. That is, there is no equilibrium measure for  $\psi$.
Consider now the flow $\Phi_2$ associated to $\tau_2|C_n:= \log 2^n$. Then
\begin{equation*}
P_{\sigma}(\Delta_{\psi} -t \tau_2)= \log \sum_{n=1}^{\infty} \frac{1}{2^{nt} \log ^2 (n+5)},
\end{equation*}
therefore there exists a unique $t^* >0$ such that $P_{\sigma}(\Delta_{\psi} -t^* \tau_2)=0$. Thus $P_{\Phi}(\psi_2)= t^*$. Moreover, if we denote by $\mu$ the equilibrium measure corresponding to
$\Delta_{\psi} -t^* \tau_1$ we have that
\begin{equation*}
\int \tau_2 \ d \mu  = \sum_{n=1}^{\infty} \frac{n\log 2}{2^{nt^*} \log ^2 (n+5)}< \infty.
\end{equation*}
Therefore, there exits a unique equilibrium measure for $\psi$
\end{example}

\begin{example}[\emph{Time changed flows such that a potential has an equilibrium measure in one flow and the other has \emph{infinite} equilibrium measure.}]
Define $\Delta_{\psi}|C_n :=-\log ((n+1) \log^2(n+1))$. Consider the flow $\Phi_1$ associated to 
$\tau_1|C_n:=2\log (n+1)$. In this case
\begin{equation*}
P_{\sigma}(\Delta_{\psi} -t \tau_1)= \log \sum_{n=1}^{\infty} \frac{1}{(n+1)^{1+2t} \log ^2(n+1)},
\end{equation*}
We thus have
\begin{equation*}
P_{\sigma}(\Delta_{\psi} -P_{\Phi_1}(\psi) \tau_1)=0,
\end{equation*}
and $P_{\Phi_1}(\psi) <1$.  Moreover, if we denote by $\mu$ the equilibrium measure corresponding to
$\Delta_{\psi} -P_{\Phi_1}(\psi) \tau_1$ we have that
\begin{equation*}
 \int \tau_1 d \mu = \sum_{n=1}^{\infty} \frac{\log^2(n+1)}{n^{2P_{\Phi_1}(\psi)+1} \log ^2 (n+1)}=\sum_{n=1}^{\infty} \frac{1}{n^{2P_{\Phi_1}(\psi)+1}}.
\end{equation*}
Since $2P_{\Phi_1}(\psi) -1 < 1$ we have that $\int \tau_1 d \mu= \infty$. Therefore $\psi$ has an infinite equilibrium measure.
Consider the flow $\Phi_2$ associated to $\tau_2|C_n:=\log 2^n,$ then (by Example \ref{ex:eqno}) there exists a unique equilibrium measure for $\psi.$
\end{example}

In the last example, we build a one parameter family of suspension semi-flow $\Phi_t$ for $t\geq 0,$ such that  for some $t^*>0,$ the map $t\mapsto h(\Phi_t)$ is constant in the interval $[0,t^*)$ and it is real analytic (and non constant) in the interval $[t^*,\infty).$

\begin{example}[Non analytic entropy map]
Choose $N \in \N$ such that
\begin{equation*}
\sum_{n>N} \frac{1}{2n \log^2(2n)} <1.
\end{equation*}
Consider the flow $\Phi$ associated to  $\tau |C_n:=  \log ( 2(n + N) \log ^2(n +N)).$
Thus,
\begin{equation*}
P_{\sigma}(-t \tau)= \log \sum_{n >N} \left(\frac{1}{2n \log ^2 (2n)}   \right)^t=
\begin{cases}
\text{infinity} & t<1; \\
\text{negative} & t >1.
\end{cases}
\end{equation*}
The entropy of the flow $\Phi$ is equal to $1$ (moreover, it has no measure of maximal entropy).
Consider now the one parameter family of roof functions given by $\tau_t(x)= \tau(x) +t$. Denote by
$\Phi_t$ the associated suspension semi-flow. Then
\begin{equation*}
h(\Phi_t)=
\begin{cases}
1 & \text{ if } P_{\sigma}(-\tau_t)\leq 0; \\
>1   & \text{ if } P_{\sigma}(-\tau_t)> 0.
\end{cases}
\end{equation*}
The map $(0,\infty)\ni t \mapsto P_{\sigma}(-\tau_t)\in \mathbb{R}$ is continuous, strictly increasing and it is negative at $t=0.$ Let be $t^*>0$ the unique zero of this map. For $ t\geq t^*$ the function $t \to h(\Phi_t)$ is real analytic (and non constant). Moreover, in  that range there exists a unique measure of maximal entropy (whereas for $ t \in [0, t^* )$ there is no measure of maximal entropy).
\end{example}

\subsection{Preservation of thermodynamic quantities.}
We have seen that the thermodynamic formalism of two flows, one obtained from a time change of the other, can be completely different. The following result establishes conditions in order for the thermodynamic formalism to be similar in a sense that would be made precise.

\begin{theorem} \label{thm:pre}
Let $\Sigma$ be a countable Markov shift. Let $\tau_1, \tau_2:\Sigma \to \R^+$ be two roof functions and $\Phi_1:Y_1 \to Y_1$ and  $\Phi_2:Y_2 \to Y_2$ the corresponding suspension  semi-flows. Consider the potentials  $\psi_2: Y_2 \to \R$ and   $\psi_1: Y_1 \to \R$  satisfying $\psi_1(x,r):= \psi_2 \circ \pi (x,r)$. If there exists a positive constant $C >0$ such that for every $x \in \Sigma$ we have
\begin{equation} \label{sim}
\frac{1}{C} < \frac{\tau_2(x)}{\tau_1(x)} < C,
\end{equation}
then
\begin{enumerate}
\item Either both systems have finite entropy or both have infinite entropy.
\item \[ \frac{P_{\Phi_1}(\psi_1)}{C} \leq P_{\Phi_2}(\psi_2) \leq CP_{\Phi_1}(\psi_1).\]
\end{enumerate}
\end{theorem}

\begin{proof}
The first claim is a consequence of the second, since $h(\Phi)=P_{\Phi}(0)$. We prove the second claim. For every positive $t \in \R$ we have $-\frac{t}{C}\tau_1 \geq -t \tau_2 \geq -tC \tau_1.$ For any $t \geq P_{\Phi_2}(\psi_2)$ we have that
\begin{equation*}
0 \geq P_{\sigma}(\Delta_{\psi_2} - t \tau_2) \geq P_{\sigma}( \Delta_{\psi_1}-tC \tau_1). 
\end{equation*}
Therefore, $P_{\Phi}(\psi_1) \leq C P_{\Phi}(\psi_2). $ On the other hand, if $t < P_{\Phi}(\psi_2)$ then
\begin{equation*}
0 \leq  P_{\sigma}(\Delta_{\psi_2} - t \tau_2) \leq  P_{\sigma}\left( \Delta_{\psi_1}-\frac{t}{C}\tau_1 \right),  
\end{equation*}
thus  for every $t < P_{\Phi_2}(\psi_2)$ we have that  $t/C \leq P_{\Phi_1}(\psi_1)$. Therefore $\frac{P_{\Phi_2}(\psi_2)}{C} \leq P_{\Phi_1}(\psi_1). $ If $P_{\Phi_1}(\psi_1)= \infty$ the so is $P_{\Phi_2}(\psi_2)$ and vice-versa.
An analogous argument holds if $t <0$.
\end{proof}

The existence of equilibrium measures is a more subtle matter and the condition given in equation \eqref{sim} is not enough to guarantee that if one potential has an equilibrium measure so does the corresponding one in the time changed flow. Indeed,

\begin{example}[\emph{Two flows with roof functions with bounded quotient, the first has a measure of maximal entropy while the second does not.}]
Let $(\Sigma, \sigma)$ be the full-shift. Consider the suspension semi-flow $\Phi_1$ with base $\Sigma$ and roof function $\tau_1 |_{C_n}:= \log (n(\log(2n))^2)$. This flow has entropy equal to one and $P_{\sigma}(-\tau_1)=0$. If we denote by $\nu_1$ the equilibrium measure corresponding to $-\tau_1$ we see that $\mu_1= \nu_1 \times m$ is the measure of maximal entropy for the flow $\Phi_1$.
Let $\Phi_2$ be the  suspension semi-flow  with base $\Sigma$ and roof function $\tau_2 |_{C_n}:= \log ((n+8)(\log(2n))^2)$ this flow has entropy equal to one and  $P_{\sigma}(-\tau_2) <0$. Therefore it does not have a measure of maximal entropy. Note that
\begin{equation*}
\lim_{n \to \infty} \frac{\tau_1|C_n}{\tau_2|C_n}=1,
\end{equation*}
In particular the roof functions satisfies equation \eqref{sim}.
\end{example}

%The above is an example which is somehow in the boundary, since the pressure of $-\tau_1$ is either infinity or non-positive.

\section{Stability results} \label{sec:sta}
This section is devoted to describe the space $\cS$ of all suspension flows defined over  a countable Markov shift $\Sigma$. We will show that the thermodynamic properties are stable or unstable depending upon the combinatorial structure of the shift space $\Sigma$. Note that the space $\cS$ can clearly be identified with the space
\[  \left\{ \tau:\Sigma \to \R: \tau \text{ is positive, locally H\"older and }  \sum_{i=0}^{\infty}\tau(\sigma^i x)=\infty  \right\}. \]
This space can be made into a topological space in the following manner (see \cite[Section 2.2]{cs}), fix an infinite sequence $\omega= (\omega_n)_n$, with $0\leq \omega_n \leq \infty$. For a potential $\phi:\Sigma \to \R$ we define
\begin{equation*}
\| \phi \|_{\omega} := \sup|\phi| + \sum_{n=1}^{\infty} \omega_n V_n(\phi), \text{ where } 0 \cdot \infty=0,
\end{equation*}
an $\epsilon-$neighbourgood of $\phi$ is given by
\begin{equation*}
B(\phi, \epsilon)= \left\{ \phi' \in \mathcal{S}: \| \phi -\phi'\|_{\omega} < \epsilon \right\}. 
\end{equation*}
The $\omega-$topology is generated by $B(\phi, \epsilon)$. If $\omega=(0,0,0, \dots)$ we obtain the sup norm, if $\omega=(0,1,1,1,1, \dots)$ we obtain the summable variation norm and if $\omega=(0,\theta^{-1}, \theta^{-2}, \dots)$ we obtain the H\"older  norm.
Note that if $\phi$ is a potential which is constant on cylinders of length one then for every $\omega$ we have that $\|\phi\|_{\omega}=\sup |\phi|.$ Also note that if $\|\phi \|_{\omega} < \epsilon$ then $\sup |\phi| < \epsilon$.
We study two cases, the full-shift and the renewal  shift, for which the pressure function is very well understood.

\subsection{The full-shift case}
We begin studying the case in which $\Sigma$ is the full-shift on a countable alphabet. The behaviour of the pressure function $t \mapsto P_{\sigma}(-t \tau)$ is completely understood in this context and when finite is a real analytic function (see subsections \ref{ss:bip} and \ref{sb:ex}).

\begin{remark}
We stress that if $\Sigma$ is the full-shift and $\phi$ is such that $P_{\sigma}(\phi)< \infty$ then
$\|\phi \|_{\omega}= \infty$ for any choice of $\omega$.
\end{remark}

The following theorem characterizes the behaviour of the flow in terms of the finitness (or not) of the entropy. Denote by $\mathcal{F} \subset \cS$ the flows with finite entropy and $\mathcal{I} \subset \cS$ those with infinite entropy.

\begin{theorem} \label{full}

The sets $\mathcal{F}$ and $\mathcal{I}$ are both open with respect to any $\omega-$topology. 
\end{theorem}

\begin{proof}
Let $(Y, \Phi)$ be a suspension semi-flow defined over the full-shift with locally H\"older roof function $\tau$ and with finite entropy. Since the system has finite entropy it means that there exists $t_0 >0 $ such that $P_{\sigma}(-t_0 \tau)= A <0$.  Recall that we are assuming that $\sum_{i=0}^{\infty}\tau(\sigma^i x)=\infty$, Kempton \cite[p.40]{keth} showed that for the full-shift  this condition implies that $\tau$ is bounded away from zero.
In particular, by ergodic optimization results (see \cite[Theorem 1]{jmu}) we have that the asymptotic slope of the pressure is bounded away from zero. That is, there exist $A<0$ such that,
\[\lim_{t \to \infty} P_{\sigma}'(-t \tau) < A < 0.\]
This implies that $\lim_{t \to \infty} P_{\sigma}(- t \tau)= -\infty$.   Let $\epsilon >0$ and consider $\tau' \in \mathcal{S}$ such that
$t_0 \| \tau -\tau'\|_{\omega}< \epsilon$. For every every $K \subset \Sigma$ compact and invariant we have that  (see \cite[Theorem 9.7 (iv)]{wa}) 
\begin{equation*}
|P_K(-t_0\tau) -P_K(-t_0 \tau')| \leq  t_0 \sup |\tau -\tau'|  \leq  \|t_0 \tau -t_0\tau'\|_{\omega} < \epsilon. \end{equation*}
It is a direct consequence of the approximation property of the Gurevich pressure (see Theorem \ref{thm:pcms}) that
\begin{equation*}
|P_{\sigma}(-t_0\tau) -P_{\sigma}(-t_0 \tau')| < \epsilon.
\end{equation*}
This implies that there exists $t_1 \in \R$ such that $P_{\sigma}(-t_1 \tau')<0$. In particular the flow
$(Y, \Phi')$ with base $(\Sigma,\sigma) $ and roof function $\tau'$ has finite entropy.  

Let us consider now $(Y', \Phi)$  a suspension semi-flow defined over the full-shift with locally H\"older roof function $\tau'$ with infinite entropy. Let  $\epsilon >0$ and $\tau_2 \in \mathcal{S}$  such that $\|\tau' -\tau_2\|_{\omega}< \epsilon$. We will show that the suspension semi-flow with base the full-shift and roof function $\tau_2$ has infinite entropy. Note that since $\sum_{i=0}^{\infty}\tau'(\sigma^i x)=\infty$ we have that $\tau'$ is bounded away from zero. Therefore, for every $t \in \R$ we have that $P_{\sigma}(-t \tau')= \infty$. Indeed, assume by way of contradiction that there exists $t'>0$ such that $P_{\sigma}(-t' \tau')< \infty$. Then ergodic optimization results (\cite[Theorem 1]{jmu}) imply that the asymptotic derivative of the pressure is a strictly negative constant. Thus, there exists $t'' >0$ such that $P_{\sigma}(-t'' \tau')< 0$. But this would imply that the entropy of the corresponding semi-flow is not infinite. Therefore, in order to prove that the suspension semi-flow with base the full-shift and roof function $\tau_2$ has infinite entropy it suffices to prove that for every $t \in \R$, $P_{\sigma}(-t\tau_2)= \infty$. Assume by way of contradiction that there exists $t_1 \in \R$ such that $P_{\sigma}(-t_1 \tau_2) < \infty$. Let $K \subset \Sigma$ be compact and invariant. Then
\begin{equation} \label{ine}
|P_{K}(-t_1 \tau') - P_{K}(-t_1 \tau_2)| \leq |t_1| \sup |\tau' - \tau_2| \leq  |t_1| \|\tau' - \tau_2 \|_{\omega} \leq |t_1| \epsilon.
\end{equation}
It follows form the approximation property of the Gurevich pressure and the fact that $P_{\sigma}(-t_1 \tau_2) < \infty$ and
$P_{\sigma}(-t_1 \tau') = \infty$ that the inequality in equation \eqref{ine} can not hold for an arbitrary compact invariant set $K$.
This contradiction proves that the set  $\mathcal{I}$ is open with respect to any $\omega-$topology. 
\end{proof}

Consider now the subset $\mathcal{C}\subset  \mathcal{F}$ of flows having measures of maximal entropy and  $\mathcal{N}\subset  \mathcal{F}$ of those not having such a measure. 

\begin{proposition}
Neither $\mathcal{C}$ nor $\mathcal{N}$ are open sets with respect to any $\omega-$topology.
\end{proposition}

\begin{proof}
There exists a locally constant (on cylinders of length one), positive function $\tau_1:\Sigma \to \R$ such that
\begin{eqnarray*}
P_{\sigma}(-t \tau_1)=
\begin{cases}
\infty & \text{ if } t < t_0;\\
0 & \text{ if } t = t_0;\\
\text{negative} & \text{ if } t > t_0.
\end{cases}
\end{eqnarray*}
and $\lim_{t \to t_0^+} P_{\sigma}'(-t \tau_1) < \infty$ (see Lemma \eqref{Lem_disc}). The suspension semi-flow $\Phi_1$ with base $(\Sigma, \sigma)$ and roof function $\tau_1$ has entropy $h(\Phi_1)=t_0$ and the equilibrium measure $\mu$ corresponding to $-t_0 \tau_1$  is such that $\int \tau_1 d \mu < \infty$ (because of the assumption on the derivative of the pressure). Therefore, the measure
\[ \frac{\mu \times m}{\int \tau_1 d \mu},\]
is the measure of maximal entropy for the flow $\Phi_1$. Let $\tau_2:\Sigma \to \R^+$ be a locally constant potential (on cylinders of length one) such that for every $x \in \Sigma$, 
 $\tau_1(x)-\tau_2(x)=\epsilon.$ Since both potentials are constant on cylinders of length one we have that $\| \tau_1 - \tau_2\|_{\omega} =
 \sup | \tau_1 - \tau_2 | \leq \epsilon$ and
  \begin{eqnarray*}
P_{\sigma}(-t \tau_2)=
\begin{cases}
\infty & \text{ if } t < t_0;\\
\text{negative} & \text{ if } t \geq t_0.
\end{cases}
\end{eqnarray*}
Therefore, the  suspension semi-flow $\Phi_2$ with base $(\Sigma, \sigma)$ and roof function $\tau_2$ has entropy $h(\Phi_2)=t_0$ and no measure of maximal entropy. This implies that the set $\mathcal{C}$ is not open with respect to any $\omega-$topology. The same argument shows that the set $\mathcal{N}$ is not open   with respect to any $\omega-$topology. 

%In order to show that the set $\mathcal{N}$ is not open   with respect to any $\omega-$topology. We proceed in a similar way considering a  locally H\"older, positive function $\tau:\Sigma \to \R$ such that
%\begin{eqnarray*}
%P_{\sigma}(-t \tau)=
%\begin{cases}
%\infty & \text{ if } t < t_0;\\
%0 & \text{ if } t = t_0;\\
%\text{negative} & \text{ if } t > t_0.
%\end{cases}
%\end{eqnarray*}
%and $\lim_{t \to t_0^+} P_{\sigma}'(-t \tau) = \infty$ (see Lemma \eqref{Lem_disc}). The corresponding suspension flow does not have a measure of maximal entropy (this is a consequence  of the assumption on the derivative of the pressure).
\end{proof}

The number $s_{\infty}(\Phi)$ is a relevant dynamical quantity of the suspension flow. It  was first used in \cite{ij} to construct examples of regular potentials exhibiting phase transitions. In \cite{irv} geodesic flows defined over non-compact manifolds of variable pinched negative curvature were studied. The manifolds considered are Hadamard manifold with an extended Schottky group (which contains parabolic elements). It turns out that these manifolds can be coded by suspension flows defined over countable Markov shifts satisfying the BIP property and with a regular roof function (see \cite[Section 4.2]{irv}). The number $s_{\infty}(\Phi)$ is related to the amount of entropy that measures escaping through the cups of the manifold can carry and with the  largest parabolic critical exponent of the extended Schottky group (see \cite[Theorem 4.12]{irv}). It is therefore of interest to understand how does this quantity varies with time changes of the flow. First note that from Lemma \ref{lem_c1} it follows that:

\begin{lemma} \label{lem_s}
Let $(\Sigma, \sigma)$ be the full-shift on a countable alphabet. For every $\alpha > 0$ there exists a positive locally H\"older potential $\tau$ such that the suspension flow $(Y, \Phi)$ with base $\Sigma$ and roof function $\tau$ satisfies $s_{\infty}(\Phi)=\alpha.$ 
\end{lemma}

We will use the following notation, if $\tau \in \cS$ we denote by $s_{\infty}(\Phi_{\tau})$ the number $s_{\infty}$ for the suspension flow $(Y, \Phi)$ with base the full-shift and roof function $\tau$. In order to study the properties of $s_{\infty}$ in $\cS$ we define an equivalence relation.  We say that $\tau, \tilde{\tau} \in \cS$ are related  if and only if  $s_{\infty}(\Phi_{\tau})=s_{\infty}(\Phi_{\tilde{\tau}})$. By Lemma \ref{lem_s} we can identify each equivalence class $[\tau]$ with a number $\alpha> 0,$ say $[\tau]=[\alpha]$ if and only if $s_{\infty}(\Phi_{\tau})=\alpha$.  We also define 
\begin{equation*}
[\infty]:=\{ \tau \in \cS: P_{\sigma}(-t\tau) =\infty \mbox{ for every } t> 0 \}.
\end{equation*}
For any $\alpha > 0,$ the class $[\alpha]$ can be partitioned in two sets 
$$
[\alpha]=\mathcal{A}^{[\alpha]}_{cont} \sqcup \mathcal{A}^{[\alpha]}_{disc}
$$
where 
$$
\mathcal{A}^{[\alpha]}_{cont}:=\{ \tau \in \cS: s_{\infty}(\Phi_{\tau})=\alpha \text{ and } -\tau \text{ is of continuous type}\}, \mbox{ and }
$$
$$
\mathcal{A}^{[\alpha]}_{disc}:=\{ \tau \in \cS: s_{\infty}(\Phi_{\tau})=\alpha \text{ and } -\tau \text{ is of discontinuous type}\}.
$$
The space of suspension flows over the full-shift can be partitioned as:
$$
\bigcup_{\alpha > 0 \vee \alpha=\infty}  [\alpha]=[\infty] \cup \bigcup_{\alpha > 0}\left(\mathcal{A}^{[\alpha]}_{cont} \sqcup \mathcal{A}^{[\alpha]}_{disc}\right)=[\infty] \sqcup \mathcal{A}_{cont} \sqcup \mathcal{A}_{disc},
$$
where $\mathcal{A}_{cont}:=\bigcup_{\alpha > 0}\mathcal{A}^{[\alpha]}_{cont}$ and  $ \mathcal{A}_{disc}:=\bigcup_{\alpha > 0}\mathcal{A}^{[\alpha]}_{disc}.$ The topological properties of this partition will describe the behaviour of the number $s_{\infty}$ in $\cS.$ 
The next result shows that if two suspension flows have different corresponding $s_{\infty}$ then their roof functions are at infinite distance with respect to any $\omega-$norm.

\begin{theorem} \label{ts} If $\tau,\tilde{\tau}\in \bigcup_{\alpha \in\mathbb{R}}[\alpha]$ and $\|\tau-\tilde{\tau} \|_w<\infty,$ then $[\tau]=[\tilde{\tau}].$
\end{theorem}

\begin{proof}
Assume that $\tau \in [\alpha], \tilde{\tau} \in [\beta]$  and $\alpha>\beta.$ If $\|\tau-\tilde{\tau} \|_w<\infty,$ then $\sup | \tau  -\tilde{\tau} | <M<\infty$ for some $M.$
Let $\epsilon>0.$ For every  $K \subset \Sigma$ compact and invariant set we have that 
$$
|P_K( -(\beta+\epsilon) \tau) -P_K( -(\beta+\epsilon) \tilde{\tau} )| \leq (\beta+\epsilon)  M.
$$

Because of the approximation property of the Gurevich pressure we have that
\begin{equation*}
|P_{\sigma}( -(\beta+\epsilon) \tau)|  \leq (\beta+\epsilon)  M+ |P_{\sigma}(-(\beta+\epsilon) \tilde{\tau} )|<\infty,
\end{equation*}
and this implies that $\alpha \leq \beta+\epsilon.$
\end{proof}

From this result, we have the following consequences that show that the  classification is stable under perturbations with respect to the norm $\| \cdot  \|_{\omega}.$

\begin{corollary}
For every $\alpha > 0,$ the set $[\alpha]$ is open with respect to any $\omega$-topology
and $diam([\alpha])=\infty$, where the diameter is computed with respect to the distance induced by the $\| \cdot \|_{\omega}$ norm.
\end{corollary}

\begin{corollary}
If $\alpha,  \beta > 0$ and $\alpha \neq \beta,$ then $dist([\alpha],[\beta])=\infty, $ with respect to the distance induced by the
$\| \cdot \|_{\omega}$ norm.
\end{corollary}

Finally, from the last inequality in the proof of the Theorem \ref{ts}, we have that

\begin{corollary}
For every number $\alpha > 0,$ the sets $\mathcal{A}^{[\alpha]}_{cont}$ and $\mathcal{A}^{[\alpha]}_{disc}$ are open with respect to any $\omega$-topology.
\end{corollary}

We conclude this sub-section describing the topological properties of the sets  having measures of maximal entropy $\mathcal{C}$ and not having such a measure $\mathcal{N}$ when restricted to the classes $\mathcal{A}_{cont}$ and $\mathcal{A}_{disc}.$

\begin{proposition}  With respect to any $\omega$-topology we have that
\begin{enumerate}
\item The set $\mathcal{C} \cap \mathcal{A}_{cont}$  is open in $\mathcal{A}_{cont}$ and $\mathcal{C} \cap \mathcal{A}_{disc}$
is not open $\mathcal{A}_{disc}.$
\item The set $\mathcal{N} \cap \mathcal{A}_{cont}$ is empty and $\mathcal{N} \cap \mathcal{A}_{disc}$ is not open in $\mathcal{A}_{disc}.$.
\end{enumerate}
\end{proposition}

\begin{proof}
From the definition of $\mathcal{A}_{cont}$ is clear that $\mathcal{C} \cap \mathcal{A}_{cont}$  is open in $\mathcal{A}_{cont}$
and that $\mathcal{N} \cap \mathcal{A}_{cont}= \emptyset$.
The proofs that the set $\mathcal{C} \cap \mathcal{A}_{disc}$
is not open $\mathcal{A}_{disc}$ and that the set $\mathcal{N} \cap \mathcal{A}_{disc}$ is not open in $\mathcal{A}_{disc}$ follow from Lemma \ref{Lem_disc} and the observations that for every $\theta>1,k\geq 2,$ given   $\epsilon>0$ small enough, we have that $\tau:=-\varphi_{\theta_{k,1},k,1}>0$ and $\tau-\epsilon>0,$ moreover $[\tau]=[\tau-\epsilon]=[1],$  
$ P_{\sigma}(-\tau)=0,$ $P_{\sigma}(-(\tau-\epsilon))=-\epsilon<0$ and $\left\| \tau- (\tau-\epsilon) \right\|_{\omega}=\epsilon.$ As $\tau\in \mathcal{C} \cap \mathcal{A}_{disc}$ and $\tau-\epsilon\in \mathcal{N} \cap \mathcal{A}_{disc},$ this implies that there are elements of $\mathcal{N} \cap \mathcal{A}_{disc}$ arbitrarily close to $\mathcal{C} \cap \mathcal{A}_{disc},$ which concludes the proof.

\end{proof}

As observed in Theorem \ref{thm_eqstate} a suspension flow $(Y, \Phi)$ with base the full-shift and roof function $\tau$ can fail to have a measure of maximal entropy for two reasons. One, which we have already studied, is that $P_{\sigma}(-h(\Phi)\tau) \neq 0$.
The other is that despite the fact that  $P_{\sigma}(-h(\Phi)\tau) = 0$, the Gibbs measure $\nu$ corresponding to $-h(\Phi)\tau$ is such that $\int h(\Phi)\tau d \nu = \infty$. This means that for any $\epsilon>0,$ the derivative of the map $t\mapsto P_{\sigma}(-t \tau)$ is unbounded in the interval $(h(\Phi),h(\Phi)+\epsilon].$ The roof functions studied in Lemma \ref{Lem_disc} provide examples of this type. For example, when $k=2 $ and $\gamma =1,$ we have that $1<\theta_{2,1}<2,$ and $\tau:=-\varphi_{\theta_{2,1},2,1}>0$ satisfies $h_{top}(\Phi_{\tau})=1,$ $P_{\sigma}(-\tau)=0$ and there exists a unique equilibrium measure $\nu$ for $-\tau$ with $\int \tau d \nu=\infty.$

\subsection{The renewal shift case}
We know consider another countable Markov shift. For the alphabet $\N\cup\{0\}$,
consider the
transition matrix $T=(t_{ij})_{i,j \in \N\cup\{0\}}$ with $t_{0,0}= t_{0,n}=
t_{n,n-1}=1$ for each $n \ge 1$ and with all other entries equal to
zero. The \emph{renewal shift} is the Markov shift
$(\Sigma_R, \sigma)$ defined by the transition matrix $T$, that~is,
the shift map $\sigma$ on the space
\[
\Sigma_R = \left\{ (x_i)_{i \ge 0} : x_i \in \N\cup\{0\} \text{ and } a_{x_i
x_{i+1}}=1 \text{ for each } i \geq 0\right\}.
\]
This shift has entropy equal to $\log 2$. The pressure function in this context is also very well understood. Indeed, Sarig \cite{Sar01a} proved the following (the version of this result for every $t \in \R$, appears in \cite[Proposition 3]{bi2}).

\begin{proposition}\label{gene-sar}
Let $(\Sigma_R, \sigma)$ be the renewal shift. For each bounded
$\phi \in \mathcal{R}$ there exist $t_{c}^+ \in (0, +\infty]$ and
 such that:
\begin{enumerate}
\item $t \mapsto P_\sigma(t\phi)$ is strictly convex and real analytic in~$(0,
t_{c}^{+})$.
\item
$P_G(t\phi)=Mt$ for
$t>t_{c}^{+}$. $M:= \sup \left\{ \int_{\Sigma}\phi \, d \mu : \mu \in
\mathcal{M} \right\}.$
\item At  $t_{c}^{+}$ the function $q \mapsto P_\sigma(t\phi)$
is continuous but not analytic.
\item For each $t \in (0,t_c^+)$ there is a unique equilibrium measure
$\mu_{t}$ for~$t\phi$.
\item For each $t >t_c^+$ there
is no equilibrium measure for $t\phi$.
\end{enumerate}
\end{proposition}

We consider now the set $\mathcal{S}$ of suspension flows with renewal shift for base.
These flows are known as \emph{renewal flows} and have been studied in \cite[Section 6]{ijt}. There exists renewal flows with infinite entropy (see  \cite[Example 6.4]{ijt}). Again, we call  $\mathcal{F} \subset \cS$ the flows with finite entropy and $\mathcal{I} \subset \cS$ those with infinite entropy. 

\begin{lemma} \label{lem_A}
The set  $\mathcal{F}$ contains an open set.
\end{lemma}

\begin{proof}
Let $\tau$ be a roof function bounded away from zero. That is, there exists $A>0$ such that for every $x \in \Sigma_R$ we have $\tau(x)>A>0$. Therefore, for every $t \geq 0$
\begin{equation*}
P_{\sigma}(-t \tau) \leq \log(2) - t A.
\end{equation*}
In particular, there exists $h >0$ such that $P_{\sigma}(-h \tau ) =0$. Thus, the suspension flow with roof function $\tau$ has finite entropy. This readily implies that the set $\mathcal{F}$ contains an open set with respect to any  $\omega-$topology. Indeed, for every $\epsilon >0$ such that $A-\epsilon >0$ and any roof function $\tau_1$ for which $\|\tau - \tau_1\|_{\omega} < \epsilon$ we have that $\tau_1 > A - \epsilon >0$. Indeed, 
\begin{equation*}
\sup \left\{	|\tau(x) -\tau_1(x)| : x \in \Sigma_R		\right\} \leq     \|\tau - \tau_1 \|_{\omega} <  \epsilon.
\end{equation*}
Therefore the suspension flow with roof function $\tau_1$ has finite entropy.
\end{proof}

\begin{theorem}
Neither $\mathcal{F}$ nor $\mathcal{I}$ are open sets with respect to any $\omega-$topology.
 \end{theorem}

\begin{proof}
Let $\tau_1$ be a roof function such that it is constant in cylinders of length one, for every $t \in \R$ we have that $P_{\sigma}(-t \tau)>0$ and $\lim_{n \to \infty} \tau_1 |C_n=0$ (for an example see  \cite[Example 6.4]{ijt}). Let
 $\tau_2$ be a roof function constant in cylinders of length one with
 \begin{equation*}
P_{\sigma}(-t\tau_2)=
\begin{cases}
\text{positive}& \text{ if } t <1;\\
0 & \text{ if } t \geq 1
\end{cases}
\end{equation*}
and $\lim_{n \to \infty} \tau_2 |C_n=0.$
Examples of this type of functions were first constructed by Hofbauer \cite{ho} (other examples can be found in \cite{it,Sar01a}). The suspension flow associated to $\tau_1$ belongs to $\mathcal{I}$ and the one corresponding to $\tau_2$ to $\mathcal{F}$.

Let $\epsilon >0$ and consider $N \in \N$ such that for any $n >N$ we have that if $x \in C_n$ then
\[ 0\leq \tau_i(x)< \frac{\epsilon}{2} \text{ for } i=1,2.\]
Define a new roof function $\tau:\Sigma_R \to \R$ by
\begin{equation*}
\tau(x):=
\begin{cases}
\tau_2(x) & \text{ if } x \in C_n \text{ for } n \in \{0, 1, \dots N\}; \\
\tau_1(x) & \text{ if } x \in C_n  \text{ for } n > N.
\end{cases}
\end{equation*}
Note that for every $t \in \R$ we have that  $P_{\sigma}(-t \tau) >0$. Indeed, the tail of $\tau$ determines the behaviour of the pressure for large values of $t$. 
Moreover
\[ \|\tau - \tau_2\|_{\omega} < \epsilon .\]
Therefore the set $\mathcal{F}$ is not open. An analogous construction shows that $\mathcal{I}$ is not open.
\end{proof}

%%%%%%%%%%%%%%%%%%%%%%%%%%%%%%%

\end{document}